\documentclass{amsart}
\newtheorem{thr}{Theorem}
\newtheorem{lem}[thr]{Lemma}

\newtheorem{stat}[thr]{Proposition}

\theoremstyle{definition}
\newtheorem{defn}[thr]{Definition}
\newtheorem{example}[thr]{Example}
\newtheorem{que}[thr]{Question}

\theoremstyle{remark}

\def\R{\mathbb{R}}

\def\op{\oplus}
\def\ot{\otimes}

\def\R{\mathbb{R}}
\def\C{\mathbb{C}}
\def\F{\mathbb{F}}

\def\trop{\textrm{trop}}
\def\Kap{\textrm{K}_\mathbb{F}}
\def\Kf{\textbf{H}_\mathbb{F}}
\def\Kfh{\textbf{H}}

\def\K{\textrm{K}}

\begin{document}

\title{On matrices with different tropical and Kapranov ranks}

\author{Yaroslav Shitov}
\address{Moscow State University, Leninskie Gory, 119991, GSP-1, Moscow, Russia}
\email{yaroslav-shitov@yandex.ru}


\begin{abstract}
In this note, we generalize the technique developed in~\cite{CJR} and prove that every $5\times n$ matrix of tropical
rank at most $3$ has Kapranov rank at most $3$, for the ground field that contains at least $4$ elements. For the ground field
either $\F_2$ or $\F_3$, we construct an example of a $5\times 5$ matrix with tropical rank $3$ and Kapranov rank $4$.
\end{abstract}

\maketitle

Tropical mathematics deals with the \textit{tropical semiring}, that is, the set $\R$ of real numbers
with the operations of tropical addition and tropical multiplication defined as $a\op b=\min\{a,b\}$ and $a\ot b=a+b$,
for all $a,b\in\R$. The connection between classical and tropical mathematics can be established with Maslov
dequantization~\cite{Lit, Mas2, Mas1}. The methods of tropical mathematics are important for different
applications~\cite{AGG, BCOQ, HOW}, and are helpful for the study of algebraic geometry~\cite{EML, Mikh}.
The notion of the rank is very interesting in tropical mathematics~\cite{AGG, DSS, BG}, and, in contrast
with the situation of matrices over a field, there are many different rank functions for tropical matrices~\cite{AGG, DSS, BG}.
This note is devoted to the concepts of the tropical and Kapranov rank functions.

We will use the symbol $\F$ to denote a field, and by $\F^*$ we will denote the set of nonzero elements of $\F$.
By $a_{ij}$ we denote an entry of a matrix $A$, by $A^{(j)}$ the $j$th column of $A$, by $A_{(i)}$ the $i$th row,
by $A^\top$ the transpose of $A$. By $A[r_1,\ldots,r_k]$ we denote the submatrix formed by the rows
of $A$ with indexes $r_1,\ldots,r_k$, and by $A[r_1,\ldots,r_k|c_1,\ldots,c_l]$ the submatrix formed by the
columns with indexes $c_1,\ldots,c_l$ of $A[r_1,\ldots,r_k]$.

By $\Kf$ we denote the field~\cite{Poon} that consists of the formal sums of the form $a(t)=\sum_{e\in\R} a_et^{e}$,
where $t$ is a variable, the \textit{coefficients} $\{a_e\}$ belong to a field $\F$, and the \textit{support} $E(a)=\{e\in\R: a_e\neq0\}$
is a well-ordered subset of $\R$ (that is, any nonempty set of $E(a)$ has the least element).
The \textit{degree} of a sum $a\in\Kf^*$ is the exponent of its \textit{leading term}, that is,
$\deg a=\min E(a).$ The element $a_0\in\F$ is called the \textit{constant term} of $a$.
We assume the degree of the zero element from $\Kf$ to equal $+\infty$. The matrix that is obtained from $A\in\left(\Kf^*\right)^{m\times n}$
by entrywise application of the mapping $\deg$ is denoted by $\deg A\in\R^{m\times n}$. Now we can define the notion of the Kapranov
rank~\cite[Corollary 3.4]{DSS}.

\begin{defn}\label{defKap}
The \textit{Kapranov rank} of a matrix $B\in\R^{m\times n}$ with respect to a \textit{ground field} $\F$
is defined to be $$\Kap(B)=\min\left\{\textrm{rank}(A)\left|\,A\in\left(\Kf^*\right)^{m\times n},\,\deg A=B\right.\right\},$$
where $rank$ is the classical rank function of matrices over the field $\Kf$.
\end{defn}

The tropical \textit{permanent} of a matrix $B\in\R^{n\times n}$ is defined to be
\begin{equation}\label{def1}\textrm{perm}(B)=\min\limits_{\sigma\in S_n} \left\{b_{1,\sigma(1)}+\ldots+b_{n,\sigma(n)}\right\},\end{equation}
where $S_n$ denotes the symmetric group on $\{1,\ldots,n\}$.
$B$ is called \textit{tropically singular} if the minimum in~(\ref{def1})
is attained at least twice. Otherwise, $B$ is called \textit{tropically non-singular}.

\begin{defn}\label{deftrop}
The \textit{tropical rank}, $\trop(M)$, of a matrix $M$$\in$$\R^{p\times q}$ is the largest number $r$
such that $M$ contains a tropically non-singular $r$-by-$r$ submatrix.
\end{defn}

The following proposition follows directly from the definitions.

\begin{stat}\label{permut}
The tropical and Kapranov ranks of a matrix remain unchanged after adding a fixed number to every element of some row or some column.
\end{stat}

For $a$ and $b$ vectors from $\left(\R\cup\{+\infty\}\right)^m$, we denote the set of all $j$ that provide the minimum
for $\min_{j=1}^m\{a_j+b_j\}$ by $\Theta(a,b)$.
The rows of a matrix $A\in\R^{m\times n}$ are called \textit{tropically linearly dependent}
(or simply \textit{tropically dependent}) if there exists $\lambda\in\R^m$ such that $\Theta(\lambda,A^{(j)})\geq 2$, for every $j\in\{1,\ldots,n\}$.
In this case, $\lambda$ is said to \textit{realize the tropical dependence} of the rows of $A$.
If the rows are not tropically dependent, then they are called \textit{tropically independent}.
The following theorem~\cite[Theorem 5.11]{Iz} plays an important role for our considerations.

\begin{thr}\label{tropthr}
The tropical rank of a matrix $A\in\R^{m\times n}$ equals the cardinality of the largest tropically independent family of rows of $A$.
\end{thr}

The present note is devoted to the following question, asked by Develin, Santos, and Sturmfels.

\begin{que}\label{que_dss}~\cite[Section 8, Question (6)]{DSS}
Is there a $5\times5$ matrix having tropical rank $3$ but Kapranov rank $4$?
\end{que}

Chan, Jensen, and Rubei~\cite[Corollary 1.5]{CJR} have shown that $\trop(A)=\K_\C(A)$, for every matrix $A\in\R^{5\times n}$.
Therefore, they answer Question~\ref{que_dss} in the most important case, the case when the Kapranov rank function is considered
with respect to a ground field $\C$. On the other hand, in the paper~\cite{DSS}, where Question~\ref{que_dss} was proposed, the Kapranov
rank was understood with respect to an arbitrary ground field~\cite[Definition 3.9]{DSS}. In our note, we consider the problem in the
case of an arbitrary field, we generalize the technique developed in~\cite{CJR} and give a general answer for Question~\ref{que_dss}.
For a field $\F$ satisfying $|\F|\geq4$ and a matrix $B\in\R^{5\times n}$ satisfying $\trop(B)\leq3$, we show that $\K_\F(B)\leq3$.
We provide examples of matrices $C$ with tropical rank $3$ satisfying $\Kap(C)=4$ if the field $\F$ is either $\F_2$ or $\F_3$.
The following lemma is helpful to prove Lemma~\ref{system2}, which gives a generalization for the technique developed in~\cite{CJR}
and holds for a more general class of ground fields.

\begin{lem}\label{forsystem2}
Let $|\F|\geq4$, $S\in\Kf^{2\times2}$. Then there exists $\xi\in\F^*$ such that
$\deg(\xi s_{i1}+s_{i2})=\min\{\deg s_{i1}, \deg s_{i2}\}$, for $i\in\{1,2\}$.
\end{lem}

\begin{proof}
If $s_{ij}\neq0$, we denote the coefficient of the leading term of $s_{ij}$ by $\sigma_{ij}$.
If $s_{ij}=0$, we choose $\sigma_{ij}\in\F^*$ arbitrarily. Now it remains to choose
$\xi\in\F\setminus\{0,-\frac{\sigma_{12}}{\sigma_{11}},-\frac{\sigma_{22}}{\sigma_{21}}\}$.
\end{proof}

\begin{lem}\label{system2}
Let $|\F|\geq4$, let a matrix $A\in\Kf^{5\times2}$ be such that $\textrm{rank}(A)=2$ and
$deg(a_{p1}a_{q2}-a_{q1}a_{p2})$$=$$\min\{\deg a_{p1}+$$\deg a_{q2},$$\deg a_{q1}$$+$$\deg a_{p2}\},$ for every different $p,q\in\{1,\ldots,5\}$.
Let also $B\in\R^{5\times n}$, denote
$$\Theta_{1j}=\Theta(\deg A^{(1)}, B^{(j)}),\,\Theta_{2j}=\Theta(\deg A^{(2)}, B^{(j)})\,\mbox{  for every $j\in\{1,\ldots,n\}$ }.$$
If $\left|\Theta_{1j}\right|\geq2$, $\left|\Theta_{2j}\right|\geq2$,
$\left|\Theta_{1j}\cup\Theta_{2j}\right|\geq3$, for every $j$, then $\K_\F(B)\leq3$.
\end{lem}

\begin{proof}
We fix an arbitrary $j\in\{1,\ldots,n\}$ and denote $\theta_1=\min_{i=1}^5\{\deg a_{i1}+b_{ij}\}$, $\theta_2=\min_{i=1}^5\{\deg a_{i2}+b_{ij}\}$.
We assume without a loss of generality that $1\in\Theta_{1j}$, $2\in\Theta_{2j}$, and that both $\Theta_1$ and $\Theta_2$ have non-empty
intersections with $\{3,4,5\}$. These settings imply that
$$\min_{\iota=3}^5\left\{\deg\det A[1,\iota]+b_{1j}+b_{\iota j}\right\}=\min_{\iota=3}^5\left\{\deg\det A[2,\iota]+b_{2j}+b_{\iota j}\right\}=\theta_1+\theta_2.$$
From Lemma~\ref{forsystem2} it then follows that there exist $\xi_3,\xi_4,\xi_5\in\F^*$ such that
\begin{equation}\label{eqsys2-1}
\deg\left(\sum_{\iota=3}^5\det A[1,\iota]t^{b_{1j}+b_{\iota j}}\xi_\iota\right)=
\deg\left(\sum_{\iota=3}^5\det A[2,\iota]t^{b_{2j}+b_{\iota j}}\xi_\iota\right)=\theta_1+\theta_2.
\end{equation}
Cramer's rule then implies that the solution $(x_1,x_2)$ of
\begin{equation}\label{eqsys2-2}\begin{cases}
a_{11}t^{b_{1j}}x_1+a_{21}t^{b_{2j}}x_2=\sum_{\iota=3}^5\xi_\iota a_{\iota1}t^{b_{\iota j}},\\
a_{12}t^{b_{1j}}x_1+a_{22}t^{b_{2j}}x_2=\sum_{\iota=3}^5\xi_\iota a_{\iota2}t^{b_{\iota j}}
\end{cases}\end{equation}
satisfies $\deg x_1=\deg x_2=0$. We set $c_{1j}=x_1t^{b_{1j}}$, $c_{2j}=x_2t^{b_{2j}}$,
$c_{\iota j}=-\xi_\iota t^{b_{\iota j}}$, for $\iota\in\{3,4,5\}$. The equations~(\ref{eqsys2-2})
imply that $\sum_{i=1}^5 a_{i1}c_{ij}=\sum_{i=1}^5 a_{i2}c_{ij}=0$. Since $j\in\{1,\ldots,n\}$
has been chosen arbitrarily, we can construct the matrix $C$ such that $B=\deg C$, and the rows
$\sum_{i=1}^5 a_{i1}C_{(i)}$ and $\sum_{i=1}^5 a_{i2}C_{(i)}$ both consist of zero elements.
From Definition~\ref{defKap} it now follows that $\K_\F(B)\leq3$.
\end{proof}

\begin{lem}\label{3ic}
Let the entries of a matrix $B\in\R^{5\times n}$ be nonnegative, every column of $B$ contain at least three zeros.
If $|\F|\geq4$, then $\K_\F(B)\leq3$.
\end{lem}

\begin{proof}
There exist different $\eta,\zeta\in\F\setminus\{0,1\}$. We set
$A=\left(\begin{smallmatrix}1&1&1&1&0\\
1&\eta&\zeta&0&1\end{smallmatrix}\right)^\top\in\Kf^{5\times2}$.
Now the result follows from Lemma~\ref{system2}.
\end{proof}

\begin{lem}\label{hardl}
Let the entries of a matrix $B\in\R^{5\times n}$ be all nonnegative, and
\begin{equation}\label{eqh1}
B=\left(
\begin{array}{c|c|c|c|c}
0\ldots0& B_1 & B_2 & B_3 & B_4 \\
0\ldots0&  &  &  \\\hline
& 0\ldots0 & 0\ldots0 & \gamma_{1}\ldots\gamma_{r} & 0\ldots0\\
B'& 0\ldots0 & \beta_{1}\ldots\beta_{q} & 0\ldots0 & 0\ldots0\\
\underbrace{\mbox{\,\,\,\,\,\,\,\,\,\,\,\,\,\,\,\,}}_{v}&\alpha_1\ldots\alpha_p&0\ldots0&0\ldots0& \underbrace{0\ldots0}_{s}\\
\end{array}
\right),
\end{equation}
where $v>0$, $p>0$, $q+r+s>0$, either $B'$ or $(B_1|\ldots|B_4)$ consists of positive numbers, and
the numbers $\alpha_1,\ldots,\alpha_p, \beta_1,\ldots,\beta_{q}, \gamma_1,\ldots,\gamma_r$ are positive.
If $|\F|\geq4$ and $\trop(B)\leq3$, then $\K_\F(B)\leq3$.
\end{lem}

\begin{proof}
The proof is by \textit{reductio ad absurdum}.

1. We assume w.l.o.g. that $B$ provides
the minimal value of $p+q+r+s$ over all matrices $D$ of the form~(\ref{eqh1})
that satisfy $\trop(D)\leq3$ and $\K_\F(D)>3$.

2. By $m$ we denote the minimal element
of the matrix $(B_1|\ldots|B_4)$. We add $-m$ to every element of the first two rows of $B$,
$m$ to every element of the first $v$ columns. So by Proposition~\ref{permut},
we can assume without a loss of generality that $B'$ consists of positive numbers and $m=0$.

3. Let each of the matrices $B_2$ and $B_3$ contain a column without zeros
(the numbers of these columns are denoted by $j_1$ and $j_2$). Items 1 and 2 show that there
exists $j_3\in\{1,\ldots,n\}$ such that either $b_{1j_3}=0$, $b_{2j_3}>0$ or $b_{1j_3}>0$, $b_{2j_3}=0$.
Then we note that the matrix $B[1,2,3,4|1,j_1,j_2,j_3]$
is tropically non-singular. Definition~\ref{deftrop} shows that the tropical rank of $B$ is not less than $4$,
so we get a contradiction. Thus we can assume without a loss of generality that every column of $B_3$ contains a zero element.

4. Theorem~\ref{tropthr} implies that there exist
$(\lambda_1,\lambda_2,\lambda_4,\lambda_5),(\mu_1,\mu_3,\mu_4,\mu_5)\in\R^4$ that
realize the tropical dependence of the rows of $B[1,2,4,5]$ and $B[1,3,4,5]$, respectively
We denote $\Lambda=(\lambda_1,\lambda_2,+\infty,\lambda_4,\lambda_5)$ and
$M=(\mu_1,+\infty,\mu_3,\mu_4,\mu_5)$, we then have that
$\Theta(\Lambda,B^{(j)})\geq2$ and $\Theta(M,B^{(j)})\geq2$, for every $j\in\{1,\ldots,n\}$.
From the equation~(\ref{eqh1}) it then follows that $\lambda_1=\lambda_2\leq\min\{\lambda_4,\lambda_5\}$,
$\mu_3=\mu_4\leq\mu_5$, and $\mu_4<\mu_1$. Now it is straightforward to check that
$\Theta(\Lambda,B^{(j)})\cup\Theta(M,B^{(j)})\geq3$, for every $j\in\{1,\ldots,n\}$.
Finally, set $A=\left(\begin{smallmatrix}t^{\lambda_{1}}&t^{\lambda_{2}}&0&t^{\lambda_{4}}&t^{\lambda_{5}}\\
t^{\mu_{1}}&0&t^{\mu_{3}}&t^{\mu_{4}}&\eta t^{\mu_{5}}\end{smallmatrix}\right)^\top\in\Kf^{5\times2}$, for some $\eta\in\F\setminus\{0,1\}$.
The application of Lemma~\ref{system2} completes the proof.
\end{proof}

Now we can prove one of the main results of this note.

\begin{thr}\label{mn}
Let $C\in\R^{5\times n}$, $\trop(C)\leq3$, and $|\F|\geq4$. Then $\K_\F(C)\leq3$.
\end{thr}

\begin{proof}
1. Theorem~\ref{tropthr} implies that the rows of $C$ are tropically dependent.
Applying Proposition~\ref{permut}, we assume without a loss of generality
that $C$ consists of nonnegative numbers, and every column of $C$ contains at least two zeros.

2. Let the minimal element of the $i$th row of $C$ is $h_i$. For every $i\in\{1,\ldots,5\}$, we add $(-h)$ to every entry of the $i$th row
of $C$, and we denote the matrix obtained by $B$. Every row of $B$ now contains at least one zero. By item~1, the entries of $B$ are nonnegative,
and every column of $B$ contains at least two zeros.

3. By Proposition~\ref{permut}, we have that $\K_\F(B)=\K_\F(C)$, $\trop(B)\leq3$.

4. If every column of $B$ contains at least three zeros, then Corollary~\ref{3ic} implies that $\trop(B)\leq3$.
So we can further assume without a loss of generality that $b_{11}=b_{21}=0$, and the elements $b_{31}, b_{41}, b_{51}$ are positive.
The three cases are possible.

\textit{Case 1.} Let some column of $B[3,4,5]$ contain exactly one zero entry. We assume without a loss of generality
that $b_{32}=0$, $b_{42}>0$, $b_{52}>0$. Assume $b_{i'j'}=0$, $b_{i''j'}>0$ for some $i',i''\in\{4,5\}$, $j'\in\{1,\ldots,n\}$.
By item 2, there exists $j''\in\{1,\ldots,n\}$ such that $b_{i''j''}=0$. We note that the matrix $B[2,3,4,5|1,2,j',j'']$
is tropically non-singular, that is, $\trop(B)\geq4$, so we get a contradiction.

Thus we see that for every $j\in\{1,\ldots,n\}$, it holds that either $b_{4j}=b_{5j}=0$ or $b_{4j},b_{5j}>0$.
So we can see that $B$ satisfies the assumptions of Lemma~\ref{hardl} up to permutations of rows and columns.

\textit{Case 2.} Assume that some column of $B[3,4,5]$ contains exactly two zero entries,
and no column of $B[3,4,5]$ contains exactly one zero entry. In this case, $B$ satisfies
the assumptions of Lemma~\ref{hardl} up to permutations of its columns.

\textit{Case 3.} Finally, we assume that for every $j\in\{1,\ldots,n\}$, it holds that either $b_{3j}=b_{4j}=b_{5j}=0$ or $b_{3j},b_{4j},b_{5j}>0$.
Let us consider the set $G$ of all $j\in\{1,\ldots,n\}$ such that the elements $b_{3j}, b_{4j}, b_{5j}$ are not all equal.
If $G$ is empty, then the last three rows of $B$ coincide, so from Proposition~\ref{permut} and Corollary~\ref{3ic}
it follows that $\Kap(B)\leq3$.

If $G$ is non-empty, then we denote $m=\min\bigcup_{g\in G}\{b_{3g},b_{4g},b_{5g}\}$.
We then add $-\min\{m,b_{3j}\}$ to every entry of the $j$th column of $B$ ($j$ runs over $\{1,\ldots,n\}$),
we also add $m$ to every element of the first two rows of $B$. We note that the matrix obtained
satisfies the conditions of Corollary~\ref{3ic}, or Case 1, or Case 2 up to permutations of its columns.
By Proposition~\ref{permut}, the matrix obtained has the same tropical and Kapranov ranks as $B$.

In any of the Cases 1--3, we see that $\K_\F(B)\leq3$. The proof is complete.
\end{proof}

Now let us show that the condition $|F|\geq4$ is necessary in the formulation of Theorem~\ref{mn}.

\begin{example}\label{left2}
Let $$B=\begin{pmatrix}1&0&0&0&1\\0&1&0&0&1\\0&0&1&0&1\\0&0&0&0&1\\1&1&1&1&0\end{pmatrix},\,\,
D=\begin{pmatrix}1&0&0&0&1\\0&1&0&0&1\\0&0&0&0&1\\0&0&0&0&1\\1&1&1&1&0\end{pmatrix}.$$
Then $\trop(B)=\trop(D)=3$, $\K_{\F_3}(B)=\K_{\F_2}(D)=4$.
\end{example}

\begin{proof}
By Definition~\ref{deftrop}, we have straightforwardly $\trop(B)=\trop(D)=3$.
Note that if a matrix $C\in\Kfh_{\F_2}^{5\times5}$ satisfies $D=\deg C$,
then $\deg\det C[1,2,3,5|1,2,3,5]=0$, so that $\K_{\F_2}(D)\geq4$. On the other hand,
the matrix $D$ contains repeating rows, thus $\K_{\F_2}(D)=4$.

Assume that a matrix $A'\in\Kfh_{\F_3}^{5\times5}$ satisfies $B=\deg A'$. Without a loss of generality
it can be assumed that $a'_{ij}=t^{b_{ij}}$, for every pair $(i,j)$ satisfying $4\in\{i,j\}$.
Note that if $\det A'[p,q,4,5|p,q,4,5]=0$ holds for every $p,q\in\{1,2,3\}$, then the entries $a_{pq}$
have the $-1$ as their constant terms, and in this case $\deg\det A'[1,2,3,5|1,2,3,5]=0$. Therefore, we
see that $\K_{\F_3}(B)\geq4$.
On the other hand, we can set $a_{ij}=t^{b_{ij}}$, for every $(i,j)\in\{1,2,3,4,5\}^2\setminus\{(4,2), (4,3)\}$, and $a_{42}=a_{43}=2+2t$,
and note that the row $A_{(2)}+A_{(3)}+A_{(4)}$ is zero, and $\deg A=B$. Definition~\ref{defKap} shows therefore that $\K_{\F_3}(B)=4$.
\end{proof}

Now we can give a general answer for Question~\ref{que_dss}.

\begin{thr}
A $5\times5$ tropical matrix $B$ with tropical rank $3$ and Kapranov rank $4$ does exist
if and only if the ground field contains at most three elements.
\end{thr}

\begin{proof}
Follows directly from Theorem~\ref{mn} and Example~\ref{left2}.
\end{proof}

I am grateful to my scientific advisor Professor Alexander E. Guterman for constant attention to my work.

\end{document}